\documentclass{amsart}

\usepackage{lineno,hyperref}
\modulolinenumbers[5]
\usepackage{amssymb}
\usepackage{amsthm,bm}
\usepackage{amsmath,amscd}
\usepackage[mathscr]{euscript}
\usepackage[all]{xy}
\usepackage[utf8]{inputenc}
\usepackage{lmodern}
\usepackage[T1]{fontenc}
\usepackage{tikz}
\usetikzlibrary{matrix,arrows,decorations.pathmorphing}

\newtheorem{theorem}{Theorem}[section]
\newtheorem{corollary}[theorem]{Corollary}
\newtheorem{lemma}[theorem]{Lemma}
\newtheorem{proposition}[theorem]{Proposition}
\theoremstyle{definition}

\theoremstyle{remark}

\numberwithin{equation}{section}

\newcommand{\tT}{\mathbf{T}}

\newcommand{\sS}{\mathbf{S}}

\newcommand{\Ch}{\mathbf{Ch}}

\DeclareMathOperator{\Cyl}{Cyl}

\DeclareMathOperator{\opp}{op}

\newcommand\adj[4]{\xymatrix{#1:#2\ar@<.5ex>[r]^{}&#3:#4\ar@<.5ex>[l]^{}}}

\newcommand{\bbZ}{\mathbb Z}
\newcommand{\bbN}{\mathbb N} 
\newcommand{\cC}{\mathcal C}
\newcommand{\frC}{\bm{\mathfrak C}}
\newcommand{\frD}{\bm{\mathfrak D}}

\newcommand{\cF}{\mathcal F}

\newcommand{\ii}{\mathcal I}
\newcommand{\jj}{\mathcal J}

\newcommand{\cM}{\mathcal M}

\newcommand{\ee}{\{e\}}

\newcommand{\OFS}{[O_\cF,\sS]}

\newcommand{\SG}{{\sS_G}}
\newcommand{\GS}{{_G\sS}}
\newcommand{\GSH}{_G\sS_H}

\begin{document}
	
	\title{A model structure via orbit spaces for equivariant homotopy}
	\thanks{This project is supported by T\"UB\.ITAK under Grant No:117F085.}
	
	\author[Mehmet Akif Erdal]{Mehmet Akif Erdal} \address{Department of Mathematics\\
		Bilkent University\\ 06800, Ankara\\ Turkey}
	\email{merdal@fen.bilkent.edu.tr}

	\author[Asl\i \ G\"u\c{c}l\"ukan \.Ilhan]{Asl\i \ G\"u\c{c}l\"ukan \.Ilhan} \address{Department of Mathematics\\
		Dokuz Eyl\"ul University\\ 06800, \.Izmir \\ Turkey}
	\email{asli.ilhan@deu.edu.tr}
	
	
	\date{\today}

\begin{abstract}
For a given group $G$ and a collection of subgroups $\mathcal{F}$ of $G$, we show that there exist a left induced model structure on the category of right $G$-simplicial sets, in which the weak equivalences and cofibrations are the maps that induce weak equivalences and cofibrations on the $H$-orbits for all $H$ in $\mathcal{F}$. This gives a model categorical criterion for maps that induce weak equivalences on $H$-orbits to be weak equivalences in the $\mathcal{F}$-model structure. 
\end{abstract}
\maketitle
\section{Introduction}

Let $\cF$ be a collection of subgroups of a given discrete group $G$ which is closed under conjugation and taking subgroups. By a space, we mean simplicial set and by a map we mean a simplicial map. There is a well-known model structure on the category $\GS$ of $G$- spaces, called the {$\cF$-model structure}, in which the weak equivalences and fibrations are maps that induce weak equivalences on $H$-fixed points for all $H\in \cF$ (firstly in \cite{piacenza} and latter generalized in \cite{stephan}). This is one of the standard ways of doing equivariant homotopy theory of (left) $G$-spaces. Let $O_\cF$ be the orbit category of $G$ with respect to the collection $\cF$; i.e., the category whose objects are homogeneous spaces $G/H$ with $H \in \cF$ and whose morphisms are $G$-equivariant maps. Here the $\cF$-model structure is transferred from the projective model structure on the category of orbit diagrams  along the left adjoint of the fixed point functor. 

Given a weak equivalence $f:A\to B$ in the $\cF$-model structure between fibrant-cofibrant objects, for each $H\in \cF$, the induced map $f/H:A/H\to B/H$ is a weak equivalence of spaces.  However, the converse statement is not true; i.e., a map that induce a weak equivalence on $H$-orbits for every $H\in \cF$, does not need to be a weak equivalence in the $\cF$-model structure. This is not true even when $\cF$ is the collection of all subgroups, in which case the model structure is the standard model structure. The following counter example is given by Tom Goodwillie (in the Math Overflow post \cite{MOGoodwillie}). Consider the action of the group $\bbZ/2$ on the suspension $\Sigma X$ by switching the two cones where $X$ is an acyclic but not a contractible space. Then both  $\Sigma X$ and the orbit space are contractible, but  $\Sigma X$ is not equivariantly contractible because the fixed point set, which is  $X$, is not contractible.

It is natural to ask when a map inducing weak equivalences on the $H$-orbits for each subgroup $H$ in $\cF$ is a weak equivalence in the $\cF$-model structure. We provide an answer to this question by constructing a model structure on $G$-spaces in which the weak equivalences and cofibrations are defined as maps inducing weak equivalences and cofibrations on the $H$-orbits (see Theorem \ref{thm:orbmodel}). For this, we apply the left transfer argument of \cite{shipleyhessinduced} to the adjoint pair $$\adj{\theta_!}{\SG}{[O_\cF,\sS]}{\theta^*},$$ where $\theta_!(A)(G/H)=A/H$ and $\theta^*(A)=A(G/\ee)$ to transfer the model structure on the category of orbit diagrams.  Therefore, we  obtain a model categorical criteria for this question: A map that induce weak equivalences on the orbits for subgroups in $\cF$ between objects that are fibrant in this model structure induces weak equivalences on $H$-fixed points for every $H$ in $\cF$ (see Corollary \ref{cor:main}). Nevertheless, the adjunction given above is not a Quillen equivalence even in the case $G=\bbZ_2$. 

In general, if the category $\cM_G$ of right $G$-objects in an admissible category $\cM$ admits underlying cofibrant replacements and  has good cylinder objects for cofibrant objects, then there exists a left induced model structure on $\cM_G$ by the left transfer argument. In the last section, we give some examples of such categories including the category of non-negatively graded chain complexes and simplicial $(G,H)$-bisets.



\section{Preliminaries}

A model category is a bi-complete category with three distinguished classes of morphisms called weak equivalences, fibrations and cofibrations satisfying certain axioms (See \cite{quillen}). A fibration (cofibration) which is also a weak equivalences is called acyclic fibration (acyclic cofibration). An object $A$ is called fibrant if the unique morphism $A \to \ast $ to the terminal object is a fibration. Similarly an object $A$ is called cofibrant if the unique morphism $\emptyset \to A$ from the initial object is a cofibration. We refer reader to \cite{hovey} and \cite[A.2.6]{lurie} for the definitions of the cofibrantly generated model categories and the combinatorial model categories.

Let $C$ be a small category and $D$ be a combinatorial model category. Then the injective model structure exists on $[C,D]$ where weak equivalences and cofibrations are defined object-wise,  see \cite[A.2.8.2]{lurie}. Moreover the injective model structure on $[C,D]$ is also combinatorial. In the following section, we use the left transfer argument given in \cite{shipleyhessinduced} to obtain a model structure on the category of $G$-spaces from the injective model structure on the covariant orbit diagrams. We include the necessary theorems here for the convenience of the reader. \\

\paragraph{\bf{The left transfer}} Let $$\adj{F}{\frD}{\frC}{G}$$ be an adjoint pair between bi-complete locally presentable categories and assume $\frC$ admits an accessible model structure.  We say $f$ is a weak equivalence (resp. cofibration) in $\frD$ if $F(f)$ is a weak equivalence (resp. cofibration) in $\frC$. If exists, this model structure on $\frD$ is called the left induced or left transferred model structure. The following is the Acyclicity theorem for  the left transfer and the condition stated in the theorem is called the acyclicity condition. 
\begin{theorem}[Corollary 3.3.4 (ii)\cite{shipleyhessinduced}] 
	Let $$\adj{F}{\frD}{\frC}{G}$$ be an adjoint pair between bi-complete locally presentable categories and assume $\frC$ admits an accessible model structure. The left induced model structure on $\frD$ exists if and only if every morphism in $\frD$ that has the left lifting property with respect to all cofibrations is a weak equivalence.
\end{theorem}
 Here, a morphism $g$ said to have the left lifting property with respect to $f$ if for every commutative square 
$$\xymatrix{
	A\ar[d]_{f} \ar[r]^{} &	C\ar[d]^{g}\\
	B\ar[r]^{}\ar@{-->}[ru]^{h} & D},$$
 there is a morphism $h$ which makes the triangles commute. It is often hard to show the acyclicity conditions  directly. However, there is a cylinder object argument, dual to the Quillen's path object argument, which implies the acyclicity.
\begin{proposition}[Theorem 2.2.1 \cite{shipleyhessinduced}] \label{prop:cylinder-object-argument}
	Given an adjoint pair $$\adj{F}{\frD}{\frC}{G}$$ between bi-complete locally presentable categories with $\frC$ an accessible model category. If \begin{enumerate}
		\item ${\frD}$ admits underlying-cofibrant replacements; i.e., for every object $A$ in ${\frD}$, there exist a weak equivalence $\varepsilon_{A}:QA\to A$ such that $QA$ is cofibrant and for every morphism $f:A\to B$ in ${\frD}$ there is a morphism  $Qf:QA\to QB$ fitting into the following commutative diagram:
		$$\xymatrix{
			QA\ar[d]_{\varepsilon_{A}} \ar[r]^{Qf} &	QB\ar[d]^{\varepsilon_{B}}\\
			A\ar[r]^{f} & B},$$
		\item for every object $A$ in ${\frD}$, its cofibrant replacement has a good cylinder object; i.e., there is a factorization  $QA\times QA  \to \Cyl(QA)\to QA$ of the codiagonal such that the first map is a cofibration and the second map is a weak equivalence. 
	\end{enumerate}
	then the left induced model structure on $\frD$ exists and is accessible. 
\end{proposition}

\section{Main results}

Given a discrete group $G$, let $BG$ be \emph{the delooping groupoid of $G$}; i.e., the groupoid having one object $\bullet$ with $BG(\bullet,\bullet)\cong G$. We denote the category of spaces  and maps by $\sS$. Let $\SG$ denote the category of right $G$-spaces with $G$-equivariant maps. One can identify $\SG$ by the functor category $[BG^{op},\sS]$.

Let $\cF$ be a given collection of subgroups of $G$ which is closed under taking subgroups and conjugates. We denote the orbit category  with respect to the collection $\cF$ by $O_\cF$, i.e., $O_\cF$ is the category whose objects are homogeneous spaces $G/H$ with $H\in \cF$ and whose morphisms are $G$-equivariant maps between them. We denote the morphism $f \in O_\cF[G/H, G/K]$ given by $f(H)=gK$ by $\widetilde{g}$. Note that $\widetilde{g} \in O_\cF[G/H,G/K]$ if and only if $H^g \subseteq K$. With the abuse of notation we write $g$ for $f$ when $H=K=\ee$. We also denote $f$ by $\widetilde{1}_K$ when $g=1$ and $H=\ee$.

Let $\theta:BG^{op}\to O_\cF$ be a functor given by $\theta(\bullet)=G/\ee$ and $\theta(g)= \widetilde{g}$. There is an induced functor $$\theta^*: \OFS\to \SG$$ defined by $\theta^*(T)=T(G/\ee)$ and $\theta^*(\eta)=\eta_{G/\ee}$ for any functor $T: O_\cF \to \sS $ and  natural transformation $\eta \in \OFS(T,U)$. Here the right $G$-action is induced by automorphisms of $G/\ee$. 
Now we define the left adjoint  $ \theta_!$ of $\theta^*$. For any object $A$ in $\SG$, let $\theta_!(A): O_\cF \to \sS$ be the functor which sends $G/H$ to the orbit space $A/H$ and $\widetilde{g}:G/H \to G/K$ to a morphism defined by 
$$ \theta_!(A)(\widetilde{g})([a]_H)=[a\cdot g]_K,$$ where $[a]_H$ denotes the equivalence class of $a$ in $A/H$. For any $f$ in $\SG$, the natural transformation $\theta_!(f)$ is the map of orbit spaces induced by $f$. 
\begin{proposition}
	The functor $ \theta_!:\SG\to \OFS$ is left adjoint to $\theta^*$.
\end{proposition}

\begin{proof}  For an object $T: O_\cF \to \sS$ and $H \leq G$, let $\varepsilon_T(G/H):T(G/\ee)/H \to T(G/H)$ be defined by 
	$$\varepsilon_T(G/H)([x]_H)=T(\widetilde{1}_H)(x)$$ 
	for every simplex $x$ in $T(G/\ee)$. Since $\widetilde{1}_H\circ h=\widetilde{1}_H:G/\ee \to G/H$, $\varepsilon_T(G/H)$ is well-defined by the functoriality of $T$. For any natural transformation $\mu:T\to T'$, $T'(\widetilde{1}_H)\circ \mu_{G/\ee}=\mu_{G/H}\circ T(\widetilde{1}_H)$; and hence, $\varepsilon:\theta_!\theta^*\Rightarrow 1_{[O_\cF,\sS]}$ is a natural transformation. Since $\theta^*\theta_!$ is the identity functor, $\theta_!$ is left adjoint to $\theta^*$. 
\end{proof}

The category $\sS$ admits a well known combinatorial model structure. Therefore the functor category $\OFS$ admits the injective model structure which is again combinatorial. Denote by $\OFS_{inj}$ the injective model structure on $\OFS$. The Acyclicity Theorem  gives us the necessary and sufficient conditions for the existence of the induced model structure along the adjunction. We use the cylinder object argument  for proving the acyclicity condition, and hence to show that the left induced model structures on $\SG$ exists.

\begin{theorem}\label{thm:orbmodel}
	There exists a left induced model structure on $\SG$, transferred along $\theta_!$ from the injective model structure on $\OFS$; i.e., $f$ in $\SG$ is a weak equivalence (resp. cofibration) if $\theta_!(f)$ is a weak equivalence  (resp. cofibration) in $\OFS_{inj}$.
\end{theorem}

\begin{proof}
	It is well-known that every object in $\sS$ is cofibrant, implying that every object in $\OFS_{inj}$ is cofibrant. Thus, $\SG$ trivially admits underlying-cofibrant replacements in the sense of \cite[2.2.2]{shipleyhessinduced}. For an object $A$ in $\SG$, let $\Cyl(A)$ be $A\times \Delta[1],$ where $\Delta[1]$ has the trivial $G$-action. Consider the factorization $$A\amalg A\stackrel{i}\to \Cyl(A)\stackrel{w}\to A,$$ where $w(x,k)=x$. Since $G$ acts trivially on $\Delta[1]$, $$  \theta_!(\Cyl(A))(G/H)= (A\times \Delta[1])/H  \cong A/H \times \Delta[1]=\theta_!(A)(G/H)\times \Delta[1]$$ for every $H\leq G$. Therefore, the following diagram commutes:
	$$\xymatrix{
		\theta_!(A \amalg  A)\ar[d]_{\cong} \ar[r]^{	\theta_!(i)} &	\theta_!(A \times \Delta[1]) \ar[d]^{\cong}\ar[r]^{\ \ \theta_!(w)}& \theta_!(A) \ar[d]^{=}\\
		\theta_!(A)\amalg \theta_!(A)\ar[r]^{\bar i}  & \theta_!(A)\times \Delta[1]\ar[r]^{\ \ \bar w}  &\theta_!(A)}$$
	Since the codiagonal in $\OFS$ is of the form $T(-)\amalg T(-)\to T(-)$, the model category $\OFS_{inj}$ admits functorial cylinder objects given by $T\mapsto T(-)\times \Delta[1]$. Thus $\bar w$ is a weak equivalence and $\bar i$ is a cofibration. By Proposition \ref{prop:cylinder-object-argument}, there is a left induced model structure on $\SG$.\end{proof}
We denote the model category obtained in Theorem \ref{thm:orbmodel} by $\SG^{{\cF}}$. It is easy to see that every object is cofibrant in $\SG^{{\cF}}$. However, the fibrant objects in the injective model does not usually have a nice descriptions, see for example \cite[p.2]{isaksen2005flasque}.

  On the other hand, $\SG^{{\cF}}$ is combinatorial. This follows directly from  \cite[2.4]{makkai2014cellular} and the fact that $\sS$, and hence $\SG$, is locally finitely presentable \cite{adamek}. Every cofibration in this model structure is clearly a monomorphism in each simplicial degree. Given an inclusion $f:A\to B$ of $G$-sets, it is straightforward that $f/H:A/H\to B/H$ is an inclusion; i.e., $f$ is a cofibration. In fact, if $[f(a)]_H=[f(a')]_H$ for some $a,a'\in A$, then there exist $h\in H$ such that $f(a)=f(a')\cdot h=f(a'\cdot h)$. Then $[a]_H=[a']_H$ by injectivity of $f$.  Thus, cofibrations of the model categories of Theorem \ref{thm:orbmodel} are precisely  inclusions in each simplicial degree for any given collection $\cF$. In particular, generating cofibrations are inclusions of finite simplicial $G$-sets  (i.e., simplicial $G$-sets that are finite in each simplicial degree) and generating acyclic cofibrations are such maps that are also weak equivalences. Fibrations are those maps that have the right lifting property with respect to all weak equivalences between finite simplicial $G$-sets that are also  inclusions in each degree. Equivalently, $X$ is fibrant if and only if for every inclusion of simplicial $G$-sets $f:A\to B$ between finite simplicial $G$-sets which is also a weak equivalence, the induced map $f^*:\SG(B,X)\to \SG(A,X)$ is a surjection. This implies, in particular, $X/H$ is a Kan complex for every $H\in \cF$. In fact, if $X$ is fibrant, then $X\to *$ has the right lifting property against horn-inclusions, since each horn inclusion, considered with the trivial $G$-action, is an acyclic cofibration in the orbit model structure.

The adjoint pair $(\theta_!\dashv \theta^* )$ is a Quillen adjunction by definition. However, it does not have to be weak equivalence. Before giving a counter example, we need  the  following well-known lemma (see also \cite[1.3.16]{hovey}) which is often used in its dual form:
\begin{lemma}
	Given a Quillen pair $\adj{L}{\cC}{\cM}{R}$, if the left adjoint is a homotopical functor that reflects weak equivalences (i.e., $L$ creates weak equivalences), then $(L \dashv R)$ is a Quillen equivalence if and only if for every fibrant object $c\in \cC$ the adjunction counit	$\epsilon:LRc\to c$ is a weak equivalence.
\end{lemma}
\begin{proof}
	Suppose that the adjunction counit
	$\epsilon:LRc\to c$ is a weak equivalence for every fibrant object $c \in \mathcal{C}$.  Since $L$ preserves the weak equivalences, the derived adjunction unit 
	$$\xymatrix{ LQR(c)\ar[r]^{Lp_{Rc}} &	 LR(c)  \ar[r]^{\epsilon}& c}$$
	is a weak equivalence.   Let $j_{(-)}$ and $p_{(-)}$ be the fibrant and cofibrant replacements. Let $d\in \cC$.  As the adjunct of the adjunct of the fibrant replacement $j_{Ld}$,  the composition
	$$\xymatrix{ Ld \ar[r]^{L\eta} &	 LRL(d)  \ar[r]^{LR(j_{Ld})}& 	LR(P(L(d)))  \ar[r]^{\epsilon_{PLd}} & PLd}.$$
 is just $j_{Ld}$. Since $PL(d)$ is fibrant, $\epsilon_{PLd}$ is a weak equivalence. By the $2$-out-of-$3$ property, the composition 
	$$\xymatrix{ Ld \ar[r]^{L\eta} &	 LRL(d)  \ar[r]^{LR(j_{Ld})}& 	LR(P(L(d))) }$$
	is a weak equivalence.	Since $L$ reflects weak equivalences, the derived adjunction counit
	$$\xymatrix{ d \ar[r]^{\eta} &	 RL(d)  \ar[r]^{R(j_{Ld})}& 	R(P(L(d))) }$$
	is a weak equivalence; i.e., the pair $(L \dashv R)$ induces equivalence on homotopy categories. Thus, $(L \dashv R)$ is a Quillen equivalence. The converse is clear. 
\end{proof}

Consider the case $G=\bbZ/2$. The orbit category has two objects $G/G$ and $G/\ee$, and two non-identity morphisms $\widetilde{1}_G:G/\ee\to G/G$ and $g:G/\ee\to G/\ee$ with $g\circ g=id$ and $\widetilde{1}_G\circ g=\widetilde{1}_G$. Consider the functor $T:O(\bbZ/2)\to \sS$ given as follows: On objects $T(G/\ee)$ is the discrete simplicial set $3$ elements in each degree, $T(G/G)$ is the terminal object, and on morphisms, $T(g)=id$ and $T(\widetilde{1}_G)$ is the obvious unique map. Let $\tilde T$ be the fibrant replacement of $T$. Then, $|\tilde T(G/\ee)|$ is homotopy equivalent to the discrete space with three elements and $|\tilde T(G/G)|$ is a contractible space. Then, $|\tilde T(G/\ee)|/G$ and $|\tilde T(G/G)|$ can not be a homotopy equivalent since there is no $\bbZ/2$ action on three element set whose quotient is a singleton. In particular, the adjunction counit $\epsilon:\theta_!\theta^*\tilde T\to \tilde T$ can not be an objectwise weak equivalence. Hence, by the lemma above  $(\theta_!\dashv \theta^* )$ can not be a Quillen equivalence. Similar examples can easily be constructed for any  finite group $G$. 

\subsection{An application}
Now, we are ready to discuss our main application. There is an adjoint equivalence of categories $$\adj{id^{-}}{\SG^{\cF}}{\GS_{\cF}}{id_{-}},$$ where  an object is sent to the same object with action reversed (where $id^{-}$ is the right adjoint).  We first prove the following proposition.
\begin{proposition}
	The inverse functor $$\adj{id^{-}}{\SG^{\cF}}{\GS_{\cF}}{id_{-}},$$ sending an object to itself with the inverse action, is a Quillen pair.
\end{proposition}
\begin{proof} The $\cF$-model structure on $\GS$ is cofibrantly generated  where the generating cofibrations and generating acyclic cofibrations are $$ \ii=\{ {G/H}\times i_n:G/H\times\partial\Delta^n\to G/H\times\Delta^n \ | \  n\in \bbN, H\in \cF \}$$ 	and $$ \jj=\{G/H\times j_n:G/H\times \Lambda^i[n]\to G/H\times \Delta^n \ | \  n\in \bbN, H\in \cF \}$$ respectively. Moreover, for all $n\in \bbN$, the boundary inclusions $\partial\Delta^n\to \Delta^n$ are cofibrations  and the horn inclusions $\Lambda^i[n]\to \Delta^n$ are acyclic cofibrations in the orbit model structure when considered as $G$-maps with trivial actions. Thus, ${G/H}\times i_n$'s are cofibrations and ${G/H}\times j_n$'s are acyclic cofibrations in the orbit model structure, since these maps are $G$-equivariant inclusions. Cofibrations are closed under retracts, pushouts and transfinite compositions. Since any cofibration (resp. acyclic cofibrations) in $\GS_{\cF}$ is a retract of a morphisms obtained by a transfinite composition of pushouts of coproducts of elements in $\ii$ (resp. $\jj$), $id_{-}:\GS_{\cF}\to \SG^{\cF}$ preserves cofibrations and acyclic cofibrations. This proves the statement.
\end{proof}
 Since right Quillen functors preserves weak equivalences between fibrant objects, we have the following Corollary. 

\begin{corollary}\label{cor:main}.
	Suppose that $A$ and $B$ be fibrant in $\SG^{\cF}$. Then, for a $G$-map $f:A\to B$, if  $f/H:A/H\to B/H$ is a weak equivalence for each $H\in \cF$ then $f^H:A^H\to B^H$ is a weak equivalence for each $H\in \cF$.
\end{corollary}

A simplicial $G$-homotopy is a left homotopy with respect to the cylinder object given in the proof of Theorem \ref{thm:orbmodel}; that is, given simplicial $G$-maps $f,g:A\to B$ are $G$-homotopic if there is a simplicial $G$-map $H:A\times \Delta[1]\to B$ such that the following diagram commutes:
$$\xymatrix{
	A\times \Delta[0]\ar[dr]_{f} \ar[r]^{id\times \delta^1} &	A\times \Delta[1] \ar[d]^{H}& A\times \Delta[0]\ar[l]_{id\times \delta^0}\ar[dl]^{g}\\
	& B  & }$$

 Since every object of $\SG^{\cF}$ are already cofibrant, by \cite[4.24]{dwyer}, we have the following result, which is the Whitehead theorem for this model structure. 
\begin{corollary}\label{cor:main2}.
	Suppose that $A$ and $B$ be fibrant in $\SG^{\cF}$. Then, a $G$-map $f:A\to B$, with  $f/H:A/H\to B/H$ is a weak equivalence for each $H\in \cF$ then $f^H:A^H\to B^H$ is a simplicial $G$-homotopy equivalence for each $H\in \cF$.
\end{corollary} 

 The category of topological spaces is not an accessible model category since it is not locally presentable (see e.g. \cite{hovey}). Hence, the left transfer argument above does not work for the category of $G$-topological spaces. On the other hand, there is "a very convenient" coreflective subcategory of topological spaces, namely, the category of $\Delta$-generated spaces, which is locally presentable. We denote by $\tT_\Delta$ and refer to \cite{duggerdelta} for the details. The category  $\tT_\Delta$ admits a combinatorial model category which has the equivalent homotopy theory as the Quillen model structure on topological spaces (i.e., the inclusion of the subcategory with the coreflector is a Quillen equivalence, see \cite{haraguchi2015model}). This category contains every $CW$-complex and the cofibrant replacement in this model category is the usual $CW$-approximation, which commute (up to homotopy) with the $H$-orbit space functor for every $H\leq G$. Thus, Theorem \ref{thm:orbmodel} holds if we replace the category of simplicial sets by the category of $\Delta$-generated spaces. The proof is similar to the proof of  Theorem \ref{thm:orbmodel}, hence we omit. The corollaries of this section follows readily for this case as well. \\

\subsection{For general  model categories}
Generalizations of the model structure above are quite straightforward after the theorem above. Let $\cM$ be an accessible model category. This guarantees the existence of the injective model structure on $[O_\cF,\cM]$, which is again accessible (see \cite[Theorem 3.4.1]{shipleyhessinduced}). The category of right $G$-objects is the functor category $[BG^{\opp},\cM]$, which we simply denote by $\cM_G$.  For a subgroup  $H\leq G$ and an object $A= BG^{\opp}\to \cM$, the $H$-quotient object $A/H$ is the colimit of the restricted diagram, $BH^{\opp}\to BG^{\opp}\to \cM$ in $\cM$. We have an adjunction $$\adj{\theta_!}{\cM_G}{ [O_\cF,\cM]}{\theta^*}$$ 
where $\theta_!(A)(G/H)=A/H$ and $\theta^*(A)=A(G/\ee)$.   If $\cM_G$ 
\begin{itemize}
	\item[i)]admits underlying cofibrant replacements, and
	\item[ii)]  has good cylinder objects for cofibrant objects,
\end{itemize} then there exists a left induced model structure on $\cM_G$, created by $\theta_!$ from the injective model structure on $[O_\cF,\cM]$ by the left transfer argument. Note that, (i) holds if every object is cofibrant in $\cM$ or subgroup orbit functor $-/H:\cM_G\to \cM$ preserves underlying cofibrant replacements and (ii) holds if there exists an object $I$ (e.g. an interval object) such that for every cofibrant object $A$ in $\cM$, the cylinder object is of the form $A\times I$ and there is a natural isomorphism $\theta_!(A\times I)\cong \theta_!(A)\times I$. Then, one can show the existence of the left induced model structure on $\cM_G$ as in the proof of Theorem \ref{thm:orbmodel}.\\

\paragraph{\bf{Equivariant Joyal model structure via orbits}}
The category $\sS$ admits another model structure in which cofibrations are monomorphisms and fibrant objects are quasi-categories (i.e., simplicial sets for which every inner horn has a filler), see \cite{joyalqcat,joyal}. This model structure is also combinatorial and the cylinder object is given by the functor $-\times J$ where $J$ is the groupoid generated by single arrow, see \cite[Prop. 6.18.]{joyal2}. Every object is clearly cofibrant. Thus, for any group $G$ and a collection of subgroups $\cF$, we have a model structure on $\sS_G$ in which a map $f$ is a weak equivalence (resp. a cofibration) if $f/H$ is a weak equivalence in the Joyal model structure (resp. an inclusion in each degree) for every $H$ belonging to $\cF$. This model structure is also combinatorial. More generally, one replaces the Joyal model structure on $\sS$ by any of its left Bousfield localization.\\

\paragraph{\bf{Equivariant model structure on $G$-chain complexes by subgroup orbits}}
Let $\Ch_+$ denote the category of non-negatively graded chain complexes in $R Mod$ with the injective model structure. The weak equivalences in this model structure are quasi-isomorphisms and cofibrations are degree-wise monomorphisms. The class of fibrations are degree-wise surjections with injective kernel and fibrant objects are chain complexes of injective modules. This model structure is also combinatorial (see e.g. \cite[2.3.13]{hovey}), where generating cofibrations are cofibrations  whose domain and codomain have cardinality less than $|R|$ if $R$ is infinite or finite if $R$ is finite. Here, the cardinality of a chain complex is the cardinality of union of modules in each degree. In other words, generating cofibrations are cofibrations between $\lambda$-small object where $\lambda$ is the supremum of $|R|$ and the first infinite cardinal $\omega$ (see \cite[2.4]{makkai2014cellular}).

Denote by $\Ch_+^G$ the category of chain complexes with right $G$-actions (i.e., the functor category $[BG^{\opp},\Ch_+]$). For a $G$-chain complex $A$, the quotient chain complex $A/H$ is the degree-wise $H$-quotients of $G$-modules. We say $f:A\to B$ in $\Ch_+^G$ is a weak equivalence (resp. cofibration) if $f/H:A/H\to B/H$ is a quasi-isomorphism (resp. degree-wise monomorphism) for every $H\in \cF$. Then, every object is clearly cofibrant in $\Ch_+^G$ which means $\Ch_+^G$ admits the underlying cofibrant replacements. The cylinder object is $\Cyl(A)_n=A_n\times A_{n-1}\times A_n$ with the $G$-action induced by the $G$-action on $A$. Then, $$\Cyl(A)_n/H\cong A_n/H\times A_{n-1}/H\times A_n/H$$ and we got an isomorphism $\Cyl(A)/H\cong \Cyl(A/H)$. Hence, we obtain a model structure on $\Ch_+^G$ with these weak equivalences and cofibrations, which is also combinatorial.  \\

\paragraph{\bf{Equivariant model structures on simplicial $(G,H)$-bisets}} Another curious example is the following: Given a pair of groups $(G,H)$, we denote the category of simplicial $(G,H)$-bisets by $\GSH$; i.e., simplicial objects in $(G,H)$-bisets. From above, we get that $\GSH$ admits a model structure in which a map $f:A\to B$ is a weak equivalence (resp. cofibration) if for every $K\leq G, L\leq H$ the map $f^K/L:A^K/L\to B^K/L$ is a weak equivalence (resp. cofibration) of simplicial sets. In fact, by \cite[2.16]{stephan} if $\cF$ is the set of all subgroups of $G$, every object is cofibrant in $\GS_{\cF}$, which implies $\GSH$ admits underlying-cofibrant replacements. Moreover, the good cylinder object is of the form  $A\times  \Delta[1]$ for every object $A$ in $\GS_{\cF}$.

\bibliographystyle{spmpsci}
\bibliography{ocmsrefs}

\end{document}